\documentclass[11pt,a4paper]{amsart}
\usepackage{amssymb,amsmath,epsfig,graphics,mathrsfs}

\usepackage{fancyhdr}
\usepackage{hyperref}
\hypersetup{
 colorlinks   = true,
 urlcolor     = blue,
 linkcolor    = blue,
 citecolor   = red ,
 bookmarksopen=true
}

\usepackage{amsmath}
\usepackage{amsfonts}
\usepackage{amssymb}
\usepackage{amsthm}
\usepackage{epsfig,graphics,mathrsfs}
\usepackage{graphicx}
\usepackage{dsfont}

\usepackage[usenames, dvipsnames]{color} 

\usepackage{hyperref}

\newcommand*\MSC[1][1991]{\par\leavevmode\hbox{%
\textit{#1 Mathematical subject classification:\ }}}
\newcommand\blfootnote[1]{%
  \begingroup
  \renewcommand\thefootnote{}\footnote{#1}%
  \addtocounter{footnote}{-1}%
  \endgroup
}

\def \phi {\varphi}

\def \R {\mathbb{R}}

\newcommand{\Rn}{\mathbb R^n}
\newcommand{\Rm}{\mathbb R^m}

\newcommand{\Hn}{\mathbb H^n}
\newcommand{\Ho}{\mathbb H^1}

\newcommand{\p}{\partial}

\newcommand{\bG}{\mathbb {G}}
\newcommand{\bg}{\mathfrak g}

\newcommand{\la}{\lambda}

\numberwithin{equation}{section}

\newcommand{\beq}{\begin{equation}}
\newcommand{\bea}[1]{\begin{array}{#1} }
\newcommand{\eeq}{ \end{equation}}
\newcommand{\ea}{ \end{array}}

\newcommand{\nh}{\nabla_H}
\newcommand{\sul}{\Delta_H}
\newcommand{\nhh}{\tilde{\nabla}_H}

\newtheorem{theorem}{Theorem}[section]

\newtheorem{proposition}[theorem]{Proposition}

\numberwithin{equation}{section}

\begin{document}

\title[A sub-Riemannian maximum modulus theorem]{A sub-Riemannian maximum modulus theorem}

\blfootnote{\MSC[2020]{35R03, 35H10, 31C05}}

\date{}

\begin{abstract}
In this note we prove a sub-Riemannian maximum modulus theorem in a Carnot group. Using a nontrivial counterexample, we also show that such result is best possible, in the sense that in its statement one cannot replace the right-invariant horizontal gradient with the left-invariant one.
\end{abstract}

\author{Federico Buseghin}

\address{Department of Mathematical Sciences\\ University of Bath\\ Bath BA2 7AY\\ United Kingdom}
\email{fb588@bath.ac.uk}

\author{Nicol\`o Forcillo}

\address{Dipartimento di Matematica\\ Universit\`a di Roma Tor Vergata\\ Via della Ricerca Scientifica, 1- 00133 Roma, Italy} 

\email{nikforc6@gmail.com}

\author{Nicola Garofalo}

\address{Dipartimento d'Ingegneria Civile e Ambientale (DICEA)\\ Universit\`a di Padova\\ Via Marzolo, 9 - 35131 Padova,  Italy}
\vskip 0.2in
\email{nicola.garofalo@unipd.it}

\thanks{The third author has been supported in part by a Progetto SID (Investimento Strategico di Dipartimento): ``Aspects of nonlocal operators via fine properties of heat kernels", University of Padova, 2022. He has also been partially supported by a Visiting Professorship at the Arizona State University.}

\maketitle


\section{Introduction and statement of the result}\label{S:intro}

The maximum modulus theorem for the standard Laplacian in $\Rn$ can be formulated as follows. Consider the unit ball $B = \{x\in \Rn\mid |x|<1\}$. Suppose that $f\in C^2(B)$ solves $\Delta f = c$ in $B$, where $c\in \R$, and that furthermore $|\nabla f|\le 1$ in $B$. If $|\nabla f(0)| = 1$, then $f(x) = f(0) + \langle\omega,x\rangle$, for some $\omega \in \mathbb S^{n-1}$ (in fact, $\omega = \nabla f(0)$). One way of proving this result is the following. Let $u(x) = |\nabla f(x)|^2$. Then we have in $B$,
\begin{equation}\label{babybo}
\Delta u = 2 ||\nabla^2 f||^2 + 2 \langle \nabla f,\nabla(\Delta f)\rangle \ge 0,
\end{equation}
where we have denoted by $\nabla^2 f$ the Hessian of $f$, and by $||\nabla^2 f||^2$ the square of its Hilbert-Schmidt norm. Since by the hypothesis, we have $0\le u \le 1$ in $B$, and $u(0) = 1$, by the strong maximum principle it follows that $u\equiv 1$ in $B$. This implies $\Delta u \equiv 0$ in $B$, and it follows from \eqref{babybo} that $||\nabla^2 f||^2 \equiv 0$. Taylor's formula thus gives the desired conclusion $f(x) = f(0) + \langle \nabla f(0),x\rangle$ in $B$.

In this note we are concerned with a sub-Riemannian version of this maximum modulus theorem (for the birth of sub-Riemannian geometry the reader should see \cite{Ca}). This question has an interest in non-holonomic free boundary problems, see \cite{DGP} and \cite{Ged}.  Specifically, let $\bG$ be a stratified nilpotent Lie group, aka Carnot group (see \cite{FS} and also \cite{Be}), and consider the left- and right-invariant vector fields associated with a orthonormal basis $\{e_1,...,e_m\}$ of the bracket generating (horizontal) layer $\bg_1$ of the Lie algebra $\bg$, see Section \ref{S:back}. Given a function $f\in C^1(\bG)$, respectively denote by
\begin{equation}\label{nabla}
|\nh f|^2 = \sum_{i=1}^m (X_i f)^2,\ \ \ \ \ \ |\nhh f|^2 = \sum_{i=1}^m (\tilde X_i f)^2,
\end{equation}
the left- and right-invariant \emph{carr\'e du champ} of $f$ relative to $\{X_1,...,X_m\}$. It is clear that, if we indicate with $e\in \bG$ the group identity, since $X_i(e) = \tilde X_i(e)$ for $i=1,...,m$, we have $|\nh f(e)|^2 = |{\tilde{\nabla}}_H f(e)|^2$. But the two objects in \eqref{nabla} are substantially different, except in the trivial situation in which the function $f$ depends exclusively on the horizontal variables, see for instance \eqref{difference} below. Consider now the left-invariant horizontal Laplacian associated with the basis $\{e_1,...,e_m\}$, defined by
\begin{equation}\label{L}
\sul  = \sum_{i=1}^m X_i^2. 
\end{equation}
Thanks to the result in \cite{Ho} this operator is hypoelliptic (but not real-analytic hypoelliptic, in general). 
Given a non-isotropic gauge $\rho$, centred at the group identity $e\in \bG$, for every $r>0$ we indicate $B_r = \{p\in \bG\mid \rho(p)<r\}$. Also, a point $p\in \bG$ will be routinely identified with its logarithmic coordinates in the Lie algebra $\bg$. We will thus write $p \cong (z,\sigma)$, where $z = (z_1,...,z_m)$ will denote the variable in the horizontal, bracket generating layer of $\bg$, and $\sigma$ will indicate the bulk variable in the remaining layers, see Section \ref{S:back}. We note that, thanks to the stratification of the Lie algebra $\bg$, a linear function $f:\bG\to \R$ cannot possibly depend on the variable $\sigma$, and must therefore be of the form $f(p) = \alpha + \langle\omega,z\rangle$, for some $\alpha\in \R$ and $\omega\in \Rm$. We prove the following result.

\begin{theorem}\label{T:main}
Let $f\in C^2(B_r)$ for a certain $r>0$, and suppose that for some $c\in \R$ we have $\sul f = c$ in $B_r$ and 
\begin{equation}\label{cond}
|\nhh f|^2 \le 1.
\end{equation}
If $|{\tilde{\nabla}}_H f(e)|^2 = 1$, then $f$ depends only on the horizontal variables $z$, and in fact  
\[
f(p) = f(e) + \langle \nabla_z f(e) ,z\rangle.
\]
\end{theorem}

We emphasise  that Theorem \ref{T:main} has a best possible character, in the sense that it is false if in its statement we replace the condition \eqref{cond} with the seemingly ``more natural" left-invariant one
\begin{equation}\label{lcond}
|\nh f|^2 \le 1.
\end{equation}
We prove in fact the following nontrivial phenomenon.

\begin{proposition}\label{P:ce}
In the Heisenberg group $\mathbb H^n$ there exist non  linear solutions of the equation $\sul f = 0$ in a ball $B_r$ (for some $r>0$), such that $|\nabla_H f(e)|^2 = 1$, and for which moreover \eqref{lcond} hold.
\end{proposition}

Theorem \ref{T:main} and Proposition \ref{P:ce} underline a striking difference between Riemannian and sub-Riemannian geometry. We recall that the celebrated identity of Bochner states that on a Riemannian manifold $M$ one has for $f\in C^3(M)$
\begin{equation}\label{boe}
\Delta(|\nabla f|^2) = 2 ||\nabla^2 f||^2 + 2 \langle\nabla(\Delta f),\nabla f\rangle + 2 \operatorname{Ric}(\nabla f,\nabla f),
\end{equation}
where $\operatorname{Ric}(\cdot,\cdot)$ indicates the Ricci tensor on $M$, see e.g. \cite[Sec. 4.3 on p.18]{CN}. This implies in particular that if $\Delta f = c$ for some $c\in \R$, and $\operatorname{Ric}(\cdot,\cdot)\ge 0$, then
\begin{equation}\label{bosub}
\Delta(|\nabla f|^2) \ge 2 ||\nabla^2 f||^2.
\end{equation}
In sub-Riemannian geometry the fundamental subharmonicity property \eqref{bosub} fails miserably, see \cite{Ged}, and also \eqref{op} in Section \ref{S:proof} below. To remedy this negative situation one must bring the right-invariant vector fields $\tilde X_i$ to center stage. 
We mention that in harmonic analysis and partial differential equations the idea of using right-invariant vector fields in a left-invariant problem has been fruitfully explored in the works \cite{MPR}, \cite{BL}, \cite{DGP}, \cite{Gmanu}, \cite{Gmanu2}, \cite{MM}, \cite{CCM}.

Concerning the results in this note we mention that, similarly to the Euclidean maximum modulus theorem, the proof of Theorem \ref{T:main} is based on a Bochner type identity for the right-invariant \emph{carr\'e du champ} $|\nhh f|^2$, see Proposition \ref{P:Br} below, and on Bony's strong maximum principle. The stratification of the Lie algebra ultimately plays again a role in concluding that $f$ must be linear. To prove Proposition \ref{P:ce}, instead, we construct a nontrivial  explicit example of a harmonic function in the Heisenberg group $\mathbb H^1$ for which the maximum modulus theorem fails for the left-invariant \emph{carr\'e du champ}, see \eqref{f} below.

This note contains three sections. Besides the present one, in Section \ref{S:back} we collect some background material that is needed in Section \ref{S:proof}, where we prove Theorem \ref{T:main} and Proposition \ref{P:ce}. 


\section{Background material}\label{S:back}

In this section we collect some background material that is needed in the rest of the paper. 
 A Carnot group of step $k\ge 1$ is a simply-connected real Lie group $(\bG, \circ)$ whose Lie algebra $\bg$ is stratified and $k$-nilpotent. This means that there exist vector spaces $\bg_1,...,\bg_k$ such that:  
\begin{itemize}
\item[(i)] $\bg=\bg_1\oplus \dots\oplus\bg_k$;
\item[(ii)] $[\bg_1,\bg_j] = \bg_{j+1}$, $j=1,...,k-1,\ \ \ [\bg_1,\bg_k] = \{0\}$.
\end{itemize}
We assume that $\bg$ is endowed with a
scalar product $\langle\cdot,\cdot\rangle$ with respect to which the layers $\bg_j's$, $j=1,...,r$,
are mutually orthogonal. We let $m_j =$ dim$\ \bg_j$, $j=
1,...,k$, and denote by $N = m_1 + ... + m_k$ the topological
dimension of $\bG$. From the assumption (ii) on the Lie
algebra it is clear that any basis of the first layer $\bg_1$ bracket generates the whole of $\bg$. Because of such special role $\bg_1$ is usually called the horizontal
layer of the stratification. For ease of notation we henceforth write $m = m_1$. 

Given $\xi, \eta\in \bg$, the Baker-Campbell-Hausdorff formula reads
\begin{equation}\label{BCH}
\exp(\xi) \circ \exp(\eta) = \exp{\bigg(\xi + \eta + \frac{1}{2}
[\xi,\eta] + \frac{1}{12} \big\{[\xi,[\xi,\eta]] -
[\eta,[\xi,\eta]]\big\} + ...\bigg)},
\end{equation}
where the dots indicate commutators of order four and higher, see \cite[Sec. 2.15]{V}. Since by (ii) above all commutators of order $k$ and higher are trivial, in every Carnot group the series in the right-hand side of \eqref{BCH} is finite. 
Notice that \eqref{BCH} assignes a group law in $\bG$, which is noncommutative when $k>1$. If $p = \exp \xi,  p' = \exp \xi'$, the group law $p \circ p'$ in $\bG$ is obtained from \eqref{BCH} by the algebraic commutation relations between the elements of its Lie algebra. 

Carnot groups are naturally equipped with a one-parameter family of automorphisms $\{\delta_\lambda\}_{\lambda>0}$  which are called the group dilations. One first defines a family of non-isotropic dilations $\Delta_\lambda :\bg \to \bg$ in the Lie algebra by assigning the formal degree $j$ to the $j$-th layer $\bg_j$ in the stratification of $\bg$. This means that if $\xi = \xi_1 + ... + \xi_k \in \bg$, with $\xi_j\in \bg_j$, $j = 1,...,k,$ one lets
\begin{equation}\label{dilg}
\Delta_\lambda \xi = \Delta_\lambda (\xi_1 + \cdots  + \xi_k) = \lambda \xi_1 + \cdots +  \lambda^k \xi_k.
\end{equation}
One then uses the exponential mapping to push forward \eqref{dilg} to the group $\bG$, i.e., we define a one-parameter family of group automorphisms $\delta_\la :\bG \to \bG$ by the equation
\begin{equation}\label{dilG}
\delta_\lambda(p) = \exp \circ \Delta_\lambda \circ \exp^{-1}(p),\quad\quad\quad p\in \bG.
\end{equation}
We will denote by $\rho$ a fixed homogeneous norm in $\bG$, see p. 8 in \cite{FS}, and set 
\begin{equation}\label{balls}
B_r = \left\{p\in \bG \mid \rho(p)<r\right\}.
\end{equation}
Notice that, since $\rho(\delta_\la p) = \la \delta(p)$, we infer that for every $p\in B_r$ one has $\delta_\la p\in B_r$ for every $0\le \la \le 1$. This means that $B_r$ is starlike with respect to the group identity $e\in \bG$, see \cite{DGS}. This simple, yet basic property, will be used in the proof of Theorem \ref{T:main} below.

Given a Carnot group $(\bG,\circ)$, we denote the left-translation operator by $L_p(p') = p \circ p'$ and with $dL_p$ its differential. The right-translation will be denoted by $R_p(p') = p' \circ p$, and its differential by $dR_p$. If we fix an orthonormal basis $\{e_1,...,e_m\}$ of the horizontal layer $\bg_1$, then we can define respectively left- and right-invariant vector fields by the formulas
\begin{equation}\label{Xi}
X_i(p) = dL_p(e_i),\ \ \ \ \ \ \ \tilde X_i(p) = dR_p(e_i).
\end{equation}
As we have mentioned, the two objects $|\nhh f|^2$ and $|\nh f|^2$ differ substantially. For instance, in the special case in which $\bG$ is a group of step $k=2$, with group constants $b^{\ell}_{ij}$, and (logarithmic) coordinates $p = (z_1,...,z_m,\sigma_1,...,\sigma_{m_2})$, one has
\begin{equation}\label{difference}
|\nh f|^2- |\nhh f|^2 =  2 \sum_{\ell=1}^{m_2} \left(\sum_{1\le i<j\le m} b^{\ell}_{ij} \left(z_i \p_{z_j} f - z_j \p_{z_{i} }f\right)
\right) \p_{\sigma_{\ell}} f,
\end{equation}
see \cite[Lemma 2.3]{Gmanu2}.
More in general, for any $\zeta \in \bg$ we respectively indicate with $Z$,
and $\tilde Z$ the left- and right-invariant vector fields on
$\bG$ defined by the Lie formulas 
\begin{equation}\label{i1}
Zf(p)= \frac{d}{dt}\ f(p\circ \exp(t\zeta))\big|_{t=0},\quad\quad\quad \ \tilde{Z} f(p) = \frac{d}{dt}\ f(
\exp(t\zeta)\circ p)\big|_{t=0} .
\end{equation}
For any $\eta , \zeta \in \bg$, for the corresponding
vector fields on $\bG$ we have the following basic commutation identities 
\begin{equation}\label{i2}
[Y , \tilde Z] = [\tilde Y , Z] = 0,
\end{equation}
see \cite{DGP}.
Such identities can be easily verified using \eqref{i1} and the Baker-Campbell-Hausdorff formula. 
From \eqref{i2} we have in particular $[X_i,\tilde X_j] = 0$, for $i, j=1,...,m$. This fact will be used in the proof of Proposition \ref{P:Br} below.

Denote now by $\mathscr Z$ the infinitesimal generator of the dilations $\{\delta_\la\}_{\la>0}$. We recall that this vector field acts on a function $f:\bG\to \R$ according to the Lie formula
\begin{equation}\label{Z}
\mathscr Z f(p) = \lim_{\lambda\to 1} \frac{f\big(\delta_\lambda (p)\big) - f(p)}{\lambda - 1} ,\quad\quad\quad p\in \bG. 
\end{equation}
Using \eqref{dilG} and \eqref{Z}, it is easy to recognise that, if we identify $p = \exp((z,\sigma))\in \bG$ with its logarithmic coordinates $(z_1,...,z_m,\sigma_{2,1},...,\sigma_{2,m_2},...,\sigma_{k,1},...,\sigma_{k,m_k})$, then
the vector field $\mathscr Z$ takes the form
\begin{equation}\label{Z1}
\mathscr Z = \sum_{j=1}^m z_j \frac{\partial}{\partial z_j} + 2 \sum_{s=1}^{m_2} \sigma_{2,s} \frac{\partial}{\partial \sigma_{2,s}} + ... + k \sum_{\ell=1}^{m_k} \sigma_{k,\ell} \frac{\partial}{\partial \sigma_{k,\ell}}.
\end{equation}
From the representation \eqref{Z1} and the Baker-Campbell-Hausdorff formula one can prove the following result. For every $j=1,...,k$ denote by $\{e_{j,1},...,e_{j,m_j}\}$ an orthonormal basis of the layer $\bg_j$ of the Lie algebra, keeping in mind that when $j=1$ we write $\{e_1,...,e_m\}$. As in \eqref{Xi}, define left-invariant vector fields in $\bG$ by setting 
\[
X_{j,\ell}(p) = dL_p(e_{j,\ell}),\ \ \ \ \ \ \ \ \ \ j=1,...,k,\ \ell =1,...,m_j.
\]
Again, we denote by $\{X_1,...,X_m\}$ the family $\{X_{1,1}, ... ,X_{1,m_1}\}$.

\begin{proposition}\label{P:Z}
There exist polynomials in $\bG$, $Q_{j,\ell}$, $j=1,...,k$, $\ell =1,...,m_j$, where for each $j=1,...,k$ the function $Q_{j,\ell}$ is homogeneous of degree $j$, i.e., $\delta_\lambda Q_{j,\ell} = \lambda^j Q_{j,\ell}$,
such that
\[
\mathscr Z = \sum_{\ell=1}^m Q_{1,\ell} X_\ell + \sum_{\ell=1}^{m_2} Q_{2,\ell} X_{2,\ell} + ... + \sum_{\ell=1}^{m_k} Q_{k,\ell} X_{k,\ell}.
\]
In the logarithmic coordinates $(z,\sigma)$ one has  $Q_{1,\ell}(p) = z_\ell$, $\ell=1,...,m$. 
\end{proposition}

In particular, Proposition \ref{P:Z} shows that $\mathscr Z$ involves differentiation not just in the horizontal directions, but in any layer of the stratification of $\bg$. As a consequence, the vector field $\mathscr Z$ is not horizontal.


\section{Proof of Theorem \ref{T:main} and Proposition \ref{P:ce}}\label{S:proof}

In this section we prove Theorem \ref{T:main} and  Proposition \ref{P:ce}. We begin by recalling the following result from \cite{Ged}, see also \cite{Gmanu2}.

\begin{proposition}[Right Bochner type identity]\label{P:Br}
Let $\bG$ be a Carnot group, $f\in C^3(\bG)$,
then one has
\begin{align}\label{B1r}
\sul(|\nhh f|^2) = 2 \langle\nhh f,\nhh(\sul f)\rangle + 2 \sum_{i=1}^m |\tilde{\nabla}_H(X_i f)|^2.
\end{align}
If in particular $\sul f = c$ for some $c\in \R$, then we have
\begin{equation}\label{B2r}
\sul(|\nhh f|^2) = 2 \sum_{i=1}^m |\tilde{\nabla}_H(X_i f)|^2 \ge 0.
\end{equation}
\end{proposition}

\begin{proof}
The proof is a straightforward calculation that uses the commutation identities $[X_i,\tilde X_j] = 0$, $i, j =1,...,m$, see \eqref{i2} above. One finds
\begin{align*}
\sul(|\nhh f|^2) & = 2 \sum_{i,j=1}^m \tilde X_j(X_i X_i f) \tilde X_j f + 2 \sum_{i,j=1}^m \tilde X_j(X_i f)  \tilde X_j(X_i f)
\\
&  = 2 \sum_{j=1}^m \tilde X_j(\sul f) \tilde X_j f  + 2 \sum_{i=1}^m |\nhh(X_i f)|^2, 
\end{align*}
from which \eqref{B1r} immediately follows.

\end{proof}

We can now present the proof of the maximum modulus theorem.

\begin{proof}[Proof of Theorem \ref{T:main}]
Suppose $\sul f = c$ in $B_r$. By hypoellipticity, we know that $f\in C^\infty(B_r)$. Consider the function $u = |\nhh f|^2$. In view of \eqref{B2r} in Proposition \ref{P:Br} such function satisfies $\sul u \ge 0$. From \eqref{cond} we have $0\le u \le 1$ in $B_r$. Furthermore, the assumption $|\nhh f(e)|^2 = 1$ implies that $u(e) = 1$. By the  connectedness of $B_r$ and Bony's strong maximum principle \cite{Bo} we conclude that $u \equiv 1$ in $B_r$. With this information in hands, we return to  \eqref{B2r} and infer that in $B_r$ we must have
\begin{equation}\label{ah}
\sul u  = \sum_{i=1}^m |\tilde{\nabla}_H(X_i f)|^2 \equiv 0.
\end{equation}
Now \eqref{ah} implies $\tilde{\nabla}_H(X_i f) \equiv 0$, for every $i=1,...,m$. Since also the right-derivatives $\tilde X_j$, $j=1,...,m$, satisfy the bracket generating assumption (see \cite{FS}), we infer that for every fixed $i\in \{1,...,m\}$, we have for all first derivatives of $X_i f$, 
\[
\p_{\sigma_{j,\ell}}(X_i f)\equiv 0\ \ \ \ \text{in}\  B_r,
\]
for $j=1,...,k$, $\ell =1,...,m_j$. From the connectedness of $B_r$ we thus conclude that there exist constants $c_1,...,c_m\in \R$ such that we have in $B_r$
\begin{equation}\label{eccoli}
X_i f \equiv c_i,\ \ \text{ for every}\ \ i=1,...,m.
\end{equation}
 We claim that \eqref{eccoli} implies the desired conclusion 
\[
f(p) = f(e) + \langle \nabla_z f(e),z\rangle,
\]
for every $p\in B_r$. 
This can be seen as follows. In view of  the bracket generating property of the vector fields $X_i$, we infer from \eqref{eccoli} that for every $i=2,...,k$ and $\ell=1,...,m_i$, we have
\[
X_{i,\ell} f\equiv 0.
\] 
Inserting this information in the expression of $\mathscr Z f$ given by Proposition \ref{P:Z},  we conclude that for every $p\in B_r$ one has 
\[
\mathscr Z f(p) = \sum_{\ell=1}^m Q_{1,\ell} X_\ell f = \langle z,c\rangle,
\]
where $c = (c_1,...,c_m)$ is defined by \eqref{eccoli}. This gives $\mathscr Zf(\delta_\la p) = \la\langle z,c\rangle$. But we have observed in Section \ref{S:back} that for every $r>0$ the set $B_r$ is starlike with respect to the identity $e$. This means that for every $p\in B_r$ the trajectory of dilations $\delta_\la(p)\in B_r$ for every $0\le \la\le 1$. The fundamental theorem of calculus thus gives
\begin{equation}\label{finito}
f(p) - f(e) = \int_0^1 \frac{d}{d\la} f(\delta_\la p) d\la = \int_0^1 \frac 1\la \mathscr Zf(\delta_\la p) d\la = \langle c,z\rangle = \langle \nabla_z f(e),z\rangle,
\end{equation}
which gives the desired conclusion.

\end{proof}

Finally, we provide the

\begin{proof}[Proof of Proposition \ref{P:ce}] 
The Heisenberg group $\Hn$ is perhaps the most important example of a Lie group of step $k=2$. Consider the first group $\mathbb H^1$ endowed with the left-invariant generators of the Lie algebra given by 
\begin{equation}\label{vf}
X_1 = \p_{x} - \frac{y}2\ \p_\sigma,\ \ \ \ \ X_{2} = \p_{y} + \frac{x}2\ \p_\sigma.
\end{equation}
If we let $T = \p_\sigma$, then the only nontrivial commutator is $[X_1,X_2] = T$. In $\mathbb H^1$ we now consider the function 
 \begin{equation}\label{f}
f(x,y,\sigma) = x + 6 y \sigma - x^3 - 21 \left(x\sigma^2 - \frac 18 xy^4 + \frac 13 y^3 \sigma - \frac{1}{40} x^5\right).
\end{equation}
For the reader's understanding, we mention that $f = P_1 + P_3 -21\ P_5$, where $P_1(x,y,\sigma) = x$, $P_3(x,y,\sigma) = 6 y \sigma - x^3$, and $P_5(x,y,\sigma) = x\sigma^2 - \frac 18 xy^4 + \frac 13 y^3 \sigma - \frac{1}{40} x^5$, are three harmonic polynomials of homogeneous degree $1, 3$ and $5$. This means that they solve in $\mathbb H^1$ the equation $\sul P_i = 0$, $i=1, 3, 5$, and that $P_i(\delta_\la p) =\la^i P_i(p)$, where $\delta_\la p = (\la x,\la y, \la^2\sigma)$ are the group non-isotropic dilations, see \eqref{dilg} and \eqref{dilG} above (the class of solid harmonics in $\mathbb H^1$ was introduced in \cite{Gr}, see also \cite{GK} for multi-dimensional generalisations). From \eqref{f} and \eqref{vf}, simple computations give 
\begin{equation}\label{Xu}
\begin{cases}
X_1 f = 1 - 3 |z|^2 - 21 \sigma^2 +\frac{21}8 x^4 + \frac{49}8 y^4 + 21 x y \sigma,
\\
\\
X_2 f = 6\sigma + 3 x y - 21 \sigma |z|^2 + 7 x y^3.
\end{cases}
\end{equation}
Furthermore, we have
\begin{equation}\label{X2u}
X_1^2 f = - 6 x + \frac{21}2 x^3 + 42 y \sigma - \frac{21}2 x y^2,\ \ \ \ \ \ \ \ \ \ X_2^2 f = 6x - \frac{21}2 x^3 - 42 y \sigma + \frac{21}2 x y^2.
\end{equation}
In particular $\sul f = 0$ in $\Ho$ (this conclusion could also be derived from the fact that $f$ is the sum of three harmonic polynomials, except that we have not verified explicitly that $\sul P_i = 0$). In order to simplify the understanding of $|\nh f|^2$, keeping in mind that $z = (x,y)$, we express \eqref{Xu} in the following way
\begin{equation}\label{Xuu}
X_1 f = 1 - 3 |z|^2 - 21 \sigma^2 + g(x,y,\sigma),\ \ \ \ \ X_2 f = 6\sigma  + h(x,y,\sigma),
\end{equation}
where we have respectively denoted 
\begin{equation}\label{giacca}
g(x,y,\sigma) = \frac{21}8 x^4 + \frac{49}8 y^4 + 21 x y \sigma,\ \ \ \ \ \ h(x,y,\sigma) = 3 x y - 21 \sigma |z|^2 + 7 x y^3.
\end{equation}
Since both $g$ and $h$ vanish at the origin, it is clear that 
\begin{equation}\label{grade}
|\nh f(0,0,0)|^2 = 1.
\end{equation}
The reader should notice that 
\begin{equation}\label{gh}
g(x,y,\sigma) = O((|z|^2 +\sigma^2)^{3/2}),\ \ \ \ \ \ h(x,y,\sigma) = O(|z|^2+\sigma^2).
\end{equation}
Using \eqref{Xuu} and \eqref{gh} it is now not difficult to verify that 
\begin{align}\label{nh}
|\nh f(x,y,\sigma)|^2 & = 1 - 6 (|z|^2 +\sigma^2)+ k(x,y,\sigma),
\end{align}
where
\begin{align}\label{ko}
k(x,y,\sigma) & = 2\ g(x,y,\sigma) + g(x,y,\sigma)^2 + h(x,y,\sigma)^2 + 12\ \sigma h(x,y,\sigma) - 6\ |z|^2 g(x,y,\sigma)
\\
& - 42\ \sigma^2 g(x,y,\sigma) + 126\ |z|^2 \sigma^2 + 9\ |z|^4 + 441\ \sigma^4.
\notag
\end{align}
It is clear that $k(0,0,0) = 0$ and a moment's thought reveals that near $(0,0,0)$ we have
\begin{equation}\label{k}
k(x,y,\sigma) = O((|z|^2 +\sigma^2)^{3/2}).
\end{equation}
From \eqref{nh} and \eqref{k} it is therefore clear that there exists $\rho>0$ such that for every $0<|z|^2+\sigma^2<\rho^2$ we have  
\begin{equation}\label{Br}
|\nh f(x,y,\sigma)|^2 < 1,
\end{equation}
therefore $|\nh f|^2$ has a strict maximum in $e = (0,0,0)$ in the Euclidean ball $B^{(e)}_\rho = \{(x,y,\sigma)\in \mathbb H^1\mid |z|^2 + \sigma^2 <\rho^2\}$, and yet $f$ is not linear. This proves that Theorem \ref{T:main} cannot possibly hold if we replace \eqref{cond} with \eqref{lcond}, thus establishing  the proposition.

\end{proof}

Some final comments are in order. In view of \eqref{grade}, \eqref{Br} and of Bony's strong maximum principle, it is clear that it is not possible for the harmonic function $f$ in \eqref{f} to satisfy
\[
\sul(|\nh f|^2) \ge 0
\]
in a connected open set contained in $B^{(e)}_\rho$. In fact, we are next going to show that the opposite inequality does hold in a full neighbourhood of the origin. From \eqref{nh} we see that
\begin{align}\label{lapF}
\sul(|\nh f|^2) &= - 6 \sul(|z|^2 +\sigma^2)+ \sul k(x,y,\sigma) = - 24 - 12 |\nh \sigma|^2 + \sul k(x,y,\sigma)
\\
& = - 24 - 3 |z|^2 + \sul k(x,y,\sigma). \notag
\end{align}
From \eqref{giacca} and \eqref{ko} we observe that $k(x,y,\sigma)$ is a sum of polynomials having homogeneous degree  $\ge 4$ with respect to the non-isotropic group  dilations. Therefore, $-3|z|^{2}+\Delta_{H}k(x,y,\sigma)$ is the sum of  polynomials with homogeneous degree $\ge 2$. Consequently, by \eqref{lapF} we must have
\begin{equation}\label{op}
\sul(|\nh f|^2) \le 0,
\end{equation}
in a sufficiently small neighbourhood of the origin. For further instances of this typically sub-Riemannian negative phenomenon, and some basic implications of it, the reader should see \cite{Ged}.



\bibliographystyle{amsplain}

\end{document}